\documentclass[11pt]{article}
\usepackage{graphicx}
\usepackage{amscd, amssymb}
\usepackage{enumerate}
\usepackage{hyperref}
\usepackage{lscape}
\newenvironment{eq}{\begin{equation}}{\end{equation}}
\newenvironment{proof}{{\bf Proof}:}{\vskip 5mm }

\newtheorem{proposition}{Proposition}[subsection]
\newtheorem{lemma}[proposition]{Lemma}
\newtheorem{definition}[proposition]{Definition}
\newtheorem{theorem}[proposition]{Theorem}

\newtheorem{remark}[proposition]{Remark}

\newtheorem{problem}[proposition]{Problem}
\newtheorem{construction}[proposition]{Construction}
\newcommand{\llabel}[1]{\label{#1}}
\newcommand{\comment}[1]{}
\newcommand{\sr}{\rightarrow}

\newcommand{\nn}{{\bf N\rm}}
\newcommand{\nat}{\nn}

{\vskip
3mm}

\begin{document}
\parskip = 2mm
\begin{center}
{\bf\Large Lawvere theories and C-systems\footnote{\em 2000 Mathematical Subject Classification: 
18C10  %	Theories (e.g. algebraic theories), structure, and semantics
08C99, % general algebraic systems, other classes of algebras, none of the above but in this section
03F50, % metamathematics of constructive systems 
18C50  %	Categorical semantics of formal languages 
}}

%\keywords {Contextual categories, Lawvere theories, monads, modules}

\vspace{3mm}

{\large\bf Vladimir Voevodsky}\footnote{School of Mathematics, Institute for Advanced Study,
Princeton NJ, USA. e-mail: vladimir@ias.edu}
\vspace {3mm}

{\large\bf December 2015}  
\end{center}

\begin{abstract} In this paper we consider the class of l-bijective C-systems, i.e., C-systems for which the length function is a bijection. The main result of the paper is a construction of an isomorphism between two categories - the category of l-bijective C-systems and the category of Lawvere theories.
\end{abstract}

\tableofcontents

\subsection{Introduction}
In this paper we consider the simplest class of C-systems (contextual categories) - the ones for which the length function $l:Ob(CC)\sr \nat$ is a bijection. It turns out that such C-systems are very closely related to Lawvere theories. We construct 
functors from the category of Lawvere theories to the category of l-bijective C-systems and a functor in the opposite direction and show in Theorem \ref{2015.12.08.th1} that these functors are mutually inverse isomorphisms of the corresponding categories. We emphasize the unexpected aspect of this result which is that we obtain not simply an equivalence but an actual isomorphism of the categories. 

This is essentially the first result on C-systems that includes the description of their homomorphisms. 

Since this paper as well as other papers in the series on C-systems is expected to play a role in the mathematically rigorous construction of  the simplicial univalent representation of the UniMath language and the Calculus of Inductive Constructions and since such a construction itself can not rely on the univalent foundations the paper is written from the perspective of the Zermelo-Fraenkel formalism. 

Constructions and proofs of the paper do not use the axiom of excluded middle or the axiom of choice. 

We use the diagrammatic order in writing compositions, i.e., for $f:X\sr Y$ and $g:Y\sr Z$ we write $f\circ g$ for the composition of $f$ and $g$. 

We do not make precise the concept of a {\em universe} that we use for some of the statements of the paper. It would be certainly sufficient to assume that $U$ is a Grothendieck universe. However, it seems likely that sets $U$ satisfying much weaker conditions can be used both for the statements and for the proofs of our results.

\subsection{The category of Lawvere theories}

For $m\in\nat$ let $stn(m)=\{i\in\nat\,|\,0\le i<m\}$ be the standard set with $m$ elements.

Let 
$$Mor(F) = \cup_{n,m\in\nat}Fun(stn(n),stn(m))$$
where $Fun(X,Y)$ is the set of functions from $X$ to $Y$. We use the definition of a function given in \cite[p.81]{Bourbaki.Sets} where a function $f$ from $X$ to $Y$ is defined as a triple $(X,Y,G)$ where $G$ is an subset in $X\times Y$ satisfying the usual conditions. This means that every function has a well defined domain and codomain which makes it possible to define a category $F$ with the set of objects $\nat$ and the set of morphisms $Mor(F)$ such that for each $n$ and $m$ the set
$$F(m,n):=\{f\in Mor(F)\,|\,dom(f)=stn(m)\,{\rm and}\,codom(f)=stm(n)\}$$
equals to $Fun(stn(m),stn(m)$ and composition, when restricted to these subsets, is the composition of functions. 

For $m,n\in\nat$ let $ii_1^{m,n}:stn(m)\sr stn(m+n)$ and $ii_2^{m,n}:stn(n)\sr stn(m+n)$ be the injections of the initial segment of length $m$ and the concluding segment of length $n$.

\begin{definition}
\llabel{2015.11.24.def1}
A Lawvere theory structure on a category $T$ is a functor $L:F\sr T$ such that the following conditions hold:
\begin{enumerate}
\item $L$ is a bijection on the sets of objects,
\item $L(0)$ is an initial object of $T$,
\item for any $m,n\in\nat$ the square
$$
\begin{CD}
L(0) @>>> L(n)\\
@VVV @VVL(ii_2^{m,n}) V\\
L(m) @>L(ii_1^{m,n})>> L(m+n)
\end{CD}
$$
is a push-out square.
\end{enumerate}
A Lawvere theory is a pair $(T,L)$ where $T$ is a category and $L$ is a Lawvere theory structure on $T$. 
\end{definition}

Let us denote the set of Lawvere theory structures on $T$ by $Lw(T)$. 
\begin{problem}\llabel{2015.12.18.prob2}
For a universe $U$, to construct a category $LW(U)$ of Lawvere theories in $U$.
\end{problem}
\begin{construction}\rm\llabel{2015.12.18.constr2}
Following Lawvere \cite[p. 61]{Lawvere} we define a morphism from a Lawvere theory ${\bf T}_1=(T_1,L_1)$ to a Lawvere theory ${\bf T}_2=(T_2,L_2)$ as a functor $G:T_1\sr T_2$ such that $L_1\circ G=L_2$. We let $Hom_{LW}({\bf T}_1,{\bf T}_2)$ denote the subset in the set of functors from $T_1$ to $T_2$ that are morphisms of the Lawvere theories. 

Note that here one uses the equality rather than isomorphism of functors. The composition of morphisms is defined as composition of functors. The identity morphism is the identity functor. The associativity and the left and right unity axioms follow immediately from the corresponding properties of the composition of functors. 

We let $Ob(LW(U))$ denote the set of Lawvere theories in $U$ and $Mor(LW(U))$ the set
$$Mor(LW(U))=\coprod_{{\bf T}_1,{\bf T}_2\in Ob(LW(U))}Hom_{LW}({\bf T}_1,{\bf T}_2)$$
Together with the obvious domain, codomain, identity and composition functions the pair of sets $Ob(LW(U))$ and $Mor(LW(U))$ form a category that we denote $LW(U)$ and call the category of Lawvere theories in $U$. 

Note that a morphism from ${\bf T}_1$ to ${\bf T}_2$ in this category is not a morphism of Lawvere theories but an iterated pair $(({\bf T}_1,{\bf T}_2),G)$ where $G$ is a morphism of Lawvere theories. However, there is an obvious bijection
$$Mor_{LW(U)}({\bf T}_1,{\bf T}_2)\sr Hom_{LW}({\bf T}_1,{\bf T}_2)$$
that we will use in both directions as a {\em coercion}, in the terminology of the proof assistant Coq, i.e., every time we have an expression which denotes an element of one of these sets in a position where an element of the other is expected it is replaced by its image under this bijection. 
\end{construction}
We will use below the following lemma.
\begin{lemma}
\llabel{2015.11.28.l4}
Let $T$ be a category and $L:F\sr T$ a functor such that the following conditions hold:
\begin{enumerate}
\item $L(0)$ is an initial object of $T$,
\item for any $m\in\nat$ the square
\begin{eq}\llabel{2015.11.28.eq2}
\begin{CD}
L(0) @>>> L(1)\\
@VVV @VVL(ii^{m,1}_2) V\\
L(m) @>L(ii^{m,1}_1)>> L(m+1)
\end{CD}
\end{eq}
is a push-out square.
\end{enumerate}
The for any $m,n\in\nat$ the square 
$$
\begin{CD}
L(0) @>>> L(n)\\
@VVV @VVL(ii_2^{m,n}) V\\
L(m) @>L(ii_1^{m,n})>> L(m+n)
\end{CD}
$$
is a push-out square.
\end{lemma}
\begin{proof}
Let $m,n\in\nat$. Consider first the diagram
$$
\begin{CD}
L(0) @>>> L(n) @>L(ii^{m,n}_2)>> L(m+n)\\
@VVV @VVL(ii^{n,1}_1)V @VVL(ii^{m+n,1}_1)V\\
L(1) @>L(ii^{n,1}_2)>> L(n+1) @>L(ii^{m,n+1}_2)>> L(m+n+1)
\end{CD}
$$
Then first square is the reflection relative to the diagonal of a square of the form  \ref{2015.11.28.eq2} and therefore it is a push-out square. 

We have $ii^{n,1}_2\circ  ii^{m,n+1}_2=ii^{m+n,1}_2$. Therefore the large square  is the reflection relative to the diagonal of a square of the form  \ref{2015.11.28.eq2} and therefore it is a push-out square. 

The right hand side square is commutative. 

By the general properties of push-out squares which are obtained from the similar properties of the pull-back squares by inversion of the direction of arrows we conclude that the right hand side square is a push-out square.

To prove the lemma proceed now by induction on $n$.

For $n=0$ the horizontal arrows are isomorphisms and since the square commutes it is a push-out square.

For $n=1$ it is a square of the form (\ref{2015.11.28.eq2}).

For the successor consider the diagram
$$
\begin{CD}
L(0) @>>> L(n) @>L(ii^{n,1}_1)>> L(n+1)\\
@VVV @VVL(ii^{m,n}_2)V @VVL(ii^{m,n+1}_2) V\\
L(m) @>L(ii^{m,n}_1)>> L(m) @>L(ii^{m+n,1}_1)>> L(m+n+1)
\end{CD}
$$
The first square is push-out by the inductive assumption. The second square is push-out by the first part of the proof. Therefore the ambient square is push-out. Since $ii^{m,n}_1\circ ii^{m+n,1}_1=ii^{m,n+1}_1$ this completes the proof of the lemma.
\end{proof}

\subsection{The category of l-bijective C-systems}

For the definition of a $C$-system see \cite{Cartmell0}, \cite{Cartmell1} (where they are called contextual categories) as well as \cite{Csubsystems}. A C-system structure on a category $CC$ is a six-tuple $cs=(l,pt,ft,p,q,s)$ where $l,ft$ and $p$ are functions, $pt$ an element of $Ob(CC)$, $q$ a partial function on pairs and $s$ a partial function. To be a C-system structure these objects must satisfy the conditions of \cite[Definitions 2.1 and 2.3]{Csubsystems}. 
\begin{definition}
\llabel{2015.11.24.def2}
A l-bijective C-system is a C-system such that the length function $Ob(CC)\sr\nat$ is a bijection.
\end{definition}
We let $Cs_{\nat}(CC)$ denote the set of l-bijective C-system structures on the category $CC$. 
\begin{problem}\llabel{2015.12.18.prob3}
Let $U$ be a universe. To construct a category $CS(U)$ of C-systems in $U$.
\end{problem}
\begin{construction}\rm
A morphism of C-systems is a functor between the underlying categories that is compatible with the corresponding C-system structures. For a detailed definition see \cite[Definition 3.1]{Cfromauniverse}. We let $Hom_{CS}(CC_1,CC_2)$ denote the set of homomorphisms from the C-system $CC_1$ to the C-system $CC_2$. 

That the composition of functors that are homomorphisms is again a homomorphism is stated in \cite[Lemma 3.2]{Cfromauniverse}. That the identity functor is a homomorphism is very easy to prove. The associativity and the left and right unity axioms for the composition of homomorphisms follow directly from the similar properties of the composition of functors. 

Repeating the approach that we used with Lawvere theories we obtain the category $CS(U)$ of C-systems in $U$.
\end{construction}
We let $Cs{\nat}(U)$ denote the full subcategory in $CS(U)$ that consists of l-bijective C-systems. 

The following result shows that l-bijective C-systems are abundant in the world of C-systems.

\begin{theorem}
\llabel{2015.12.26.th1}
Let $C$ be a C-system. Let $X$ be an object of $C$ such that $l(X)=1$. Let $C_X$ be the smallest C-subsystem of $C$ that is a full subcategory of $C$ and that contains $X$. Then $C_X$ is an l-bijective C-system.
\end{theorem}
\begin{proof}
Recall that we let $pt$ denote the object of $C$ of length $0$ and for an object $Y$ of $C$ we let $\pi_Y$ denote the unique morphism $Y\sr pt$. Define objects $X^{*n}$ for $n\in\nat$ inductively as follows:
\begin{enumerate}
\item $X^{*0}=pt$, 
\item $X^{*1}=X$,
\item $X^{*(n+1)}=(\pi_{X^{*n}})^*(X)$.
\end{enumerate}
Let $C_X$ be the full subcategory of $C$ that is generated by objects $X^{*n}$ for $n\in\nat$. Let us show that it is a C-subsystem of $C$. Since it is a full subcategory, i.e., contains all morphisms between its objects it is sufficient to show that it is closed under the operations of C-systems that generate objects.

We have $pt\in C_X$.

For $X^{*n}\in C_X$ we have $ft(X^{*n})\in C_X$. This is proved by an easy induction on $n$.

For $f:X^{*m}\sr X^{*n}$ we have $f^*(X^{*(n+1)})\in C_X$. Indeed, by definition we have $X^{*(n+1)}=(\pi_{X^{*n}})^*(X)$. Therefore
$$f^*(X^{*(n+1)})=f^*((\pi_{X^{*n}})^*(X))=(f\circ \pi_{X^{*n}})^*(X)=\pi_{X^{*m}}^*(X)=X^{*(m+1)}$$
To finish the proof of the theorem it remains to show that $C_X$ is an l-bijective C-system which is straightforward. 
\end{proof}

\subsection{A functor from Lawvere theories to l-bijective C-systems}
\begin{problem}
\llabel{2015.11.24.prob1}
For a category $T$ to construct a function 
$$LC:Lw(T)\sr Cs_{\nat}(T^{op})$$
from the Lawvere theory structures on $T$ to the l-bijective C-system structures on $T^{op}$. 
\end{problem}
\begin{construction}\rm
\llabel{2015.11.24.constr1}
Let $CC=T^{op}$. We need to construct a l-bijective C-system structure on $CC$. We set:

The length function $l=L^{-1}$. 

The distinguished final object $pt$ is $L(0)$. 

The map $ft:Ob(CC)\sr Ob(CC)$ maps $pt$ to $pt$ and any object $X$ such that $l(X)>0$ to $L(l(X)-1)$.

For $pt$ the morphism $p_{pt}$ is the identity. For $X$ such that $l(X)>0$ the morphism $p_X:X\sr ft(X)$ is $L(ii_1^{l(X)-1,1})$.

To define $q(f,X)$ observe first that for any $X$ such that $l(X)>0$ we have a pull-back square in $CC$ of the form
\begin{eq}\llabel{2015.11.24.eq1}
\begin{CD}
X @>L(ii_2^{l(X)-1,1})>> L(1)\\
@VL(ii_1^{l(X)-1,1})VV @VVV\\
ft(X) @>>> L(0)
\end{CD}
\end{eq}
Given $f:Y\sr X$ we set
$$f^*(X)=L(l(Y)+1)$$
Since (\ref{2015.11.24.eq1}) is a pull-back square and $L(0)$ is a final object there is a unique morphism $q(f,X):f^*(X)\sr X$ such that 
\begin{eq}
\llabel{2015.12.02.eq7}
q(f,X)\circ p_X=p_{f^*(X)}\circ f
\end{eq}
and
\begin{eq}
\llabel{2015.12.02.eq8}
q(f,X)\circ L(ii_2^{l(X)-1,1})=L(ii_2^{l(Y),1})
\end{eq}
Let us check the conditions of \cite[Definition 2.1]{Csubsystems} that will show that we obtained a C0-system. We have $l^{-1}(0)=L(0)=\{pt\}$. For $X$ such that $l(X)>0$ we have $l(ft(X))=l(L(l(X)-1))=l(X)-1$. We also have $ft(pt)=pt$. The object $pt=L(0)$ is final. 

The square
\begin{eq}\llabel{2015.11.24.eq2}
\begin{CD}
f^*(X) @>q(f,X)>> X \\
@Vp_{f^*(X)}VV @Vp_XVV\\
Y @>f>> ft(X)
\end{CD}
\end{eq}
commutes by the construction of $q(f,X)$. 

If $f=Id_{ft(X)}$ then $l(Y)=l(X)$ and therefore $f^*(X)=X$. Therefore $q(f,X)\circ p_X=p_X$ and $q(f,X)\circ L(ii_2^{l(X)-1,1})=L(ii_2^{l(X)-1,1})$ which proves that $q(f,X)=Id_X$.

Given $g:Z\sr Y$ we have to verify that $q(g\circ f,X)=q(g,f^*(X))\circ q(f,X)$. 
We have $(g\circ f)^*(X)=L(l(Z)+1)=g^*(f^*(X))$.  Taking into account that (\ref{2015.11.24.eq1}) is a pull-back square, it remains to verify two equalities
$$q(g\circ f,X)\circ p_X=q(g,f^*(X))\circ q(f,X)\circ p_X$$
and
$$q(g\circ f,X)\circ L(ii_2^{l(X)-1,1})=q(g,f^*(X))\circ q(f,X)\circ L(ii_2^{l(X)-1,1})$$
For the first equality we have
$$q(g\circ f,X)\circ p_X=p_{(g\circ f)^*(X)}\circ g\circ f=p_{g^*(f^*(X))}\circ g\circ f=$$$$q(g,f^*(X))\circ p_{f^*(X)}\circ f=q(g,f^*(X))\circ q(f,X)\circ p_X$$
and for the second
$$q(g\circ f,X)\circ L(ii_2^{l(X)-1,1})=L(ii_2^{l(Z),1})=q(g,f^*(X))\circ L(ii_2^{l(Y),1})=q(g,f^*(X))\circ q(f,X)\circ L(ii_2^{l(X)-1,1})$$
According to \cite[Proposition 2.4]{Csubsystems} it remains to show that the squares (\ref{2015.11.24.eq2}) are pull-back squares. 

Consider the diagram 
$$
\begin{CD}
f^*(X) @>q(f,X)>> X @>L(ii_2^{l(X)-1,1})>> L(1)\\
@Vp_{f^*(X)}VV @Vp_XVV @VVV\\
Y @>f>> ft(X) @>>> L(0)
\end{CD}
$$
where both the right hand side square and the outside square are of the form (\ref{2015.11.24.eq1}) and in particular are pull-back squares and the left hand side square has been shown to be commutative. Therefore, the left hand side square is a pull-back square. We conclude, by \cite[Proposition 2.4]{Csubsystems} that there exists a unique $s$ such that $(l,pt,p,q,s)$ is a C-system structure. 
\end{construction}
\begin{lemma}\llabel{2015.12.06.prob1}
Let $G:(T_1,L_1)\sr (T_2,L_2)$ be a morphism of Lawvere theories. Then the functor $G^{op}$ is a homomorphism of C-systems $(T_1^{op},LC(L_1))\sr (T_2^{op},LC(L_2))$. 
\end{lemma}
\begin{proof}\rm\llabel{2015.12.06.constr1}
For convenience we will write $H$ instead of $G^{op}$. In view of \cite[Lemma 3.4]{Cfromauniverse} it is sufficient to verify that $H$ is compatible with the length function, distinguished final object, $ft$ map, $p$-morphisms and $q$-morphisms. 

The fact that it $l_1=H\circ l_2$ is equivalent to the fact that $L_1\circ G=L_2$. 

The fact that $H(L_1(0))=L_2(0)$ follows from the same property of $G$.

The fact that $H(ft(X))=ft(H(X))$ again follows from the same property of $G$.

The fact that $H(p_X)=p_{H(X)}$ follows from the fact that $L_1\circ G=L_2$ on objects and on morphisms of the form $ii_1^{n,1}$.

It remains to verify that for $X$ such that $l(X)>0$ and $f:Y\sr ft(X)$ one has 
$$H(q(f,X))=q(H(f),H(X))$$
where the right hand side is defined because $H$ is compatible with $l$ and $ft$. 

Morphism $q(H(f),H(X))$ is defined as the unique morphism such that 
$$q(H(f),H(X))\circ p_{H(X)}=p_{H(f)^*(H(X))}\circ H(f)$$
and
$$
q(H(f),H(X))\circ L_2(ii_2^{l_2(H(X))-1,1})=L_2(ii_2^{l_2(H(Y)),1})
$$
Therefore we need to verify the same equalities for the morphism $H(q(f,X))$. For the first equality we have
$$H(q(f,X))\circ p_{H(X)}=H(q(f,X))\circ H(p_X)=H(q(f,X)\circ p_X)=H(p_{f^*(X)}\circ f)=p_{H(f^*(X))}\circ H(f)$$
It remains to show that $H(f^*(X))=H(f)^*(H(X))$. It follows from the fact that
$f^*(X)=L_1(l_1(X)+1)$ and $H(f)^*(H(X))=L_2(l_2(H(X))+1)$.

For the second equality we have 
$$H(q(f,X))\circ L_2(ii_2^{l_2(H(X))-1,1})=H(q(f,X))\circ H(L_1(ii_2^{l_1(X)-1,1}))=H(q(f,X)\circ L_1(ii_2^{l_1(X)-1,1}))=$$$$H(L_1(ii_2^{l_1(Y),1}))=L_2(ii_2^{l_2(H(Y)),1})$$
This completes the construction. 
\end{proof}
\begin{problem}\rm\llabel{2015.12.08.prob1}
To construct, for any universe $U$, a functor $LC_U:LW(U)\sr LS_{\nat}(U)$.
\end{problem}
\begin{construction}\rm\llabel{2015.12.08.constr1}
We set $LC_{Ob}$ to be the function that takes a Lawvere theory to the opposite category of its underlying category with the C-system structure defined by Construction \ref{2015.11.24.constr1}. We set $LC_{Mor}$ to be the function that takes a functor $G$ that is a morphism of Lawvere theories to $G^{op}$. It is well defined by Lemma \ref{2015.12.06.prob1}. That the functions $(LC_{Ob},LC_{Mor})$ form a functor, i.e., commute with the identity morphisms and compositions is straightforward.
\end{construction}

\begin{remark}\rm
\llabel{2015.11.24.rem1}
The morphism $p:L(1)\sr L(0)$, pull-back squares (\ref{2015.11.24.eq1}) and the final object $L(0)$ make $T^{op}$ into a universe category $(T^{op},p)$ (see \cite{Cfromauniverse}). It is easy to prove that the C-system of Construction \ref{2015.11.24.constr1} is isomorphic to the C-system $CC(T^{op},p)$ of this universe category. However this isomorphism is not an equality since the set of objects  of the category $CC(T^{op},p)$ is not equal to the set of objects of $T^{op}$. Indeed, at object of $CC(T^{op},p)$ is a sequence of the form
$$(pt;\pi_{L(0)},\dots,\pi_{L(n)})$$
where $\pi_{L(i)}$ is the unique morphism  $L(i)\sr L(0)$. 
\end{remark}

\subsection{A functor from l-bijective C-systems to Lawvere theories}
\begin{problem}\llabel{2015.11.24.prob2}
For a category $CC$ to construct a function 
$$CL:Cs_{\nat}(CC)\sr Lw(CC^{op})$$
\end{problem}
To perform a construction we will need a number of lemmas and intermediate constructions. Let us fix a category $CC$ and a l-bijective C-system structure $cs=(l,pt,ft,p,q,s)$ on $CC$. We will often write $CC$ both for the category and for the C-system $(CC,cs)$. 
\begin{problem}
For $n\in\nat$ and $i=0,\dots,n-1$ to construct a morphism $\pi^n_i:l^{-1}(n)\sr l^{-1}(1)$ in $CC$. 
\end{problem}
\begin{construction}
\rm\llabel{2015.11.28.constr1}
By induction on $n$. 

For $n=0$ there are no morphisms to construct. 

For $n=1$ we set $\pi^1_0=Id_{l^{-1}(1)}$.

For the successor consider the canonical square:
\begin{eq}
\llabel{2015.11.28.eq1}
\begin{CD}
l^{-1}(n+1) @>q(\pi,l^{-1}(1))>> l^{-1}(1)\\
p_{l^{-1}(n+1)}@VVV @VVVp_{l^{-1}(1)}\\
l^{-1}(n)@>\pi>>l^{-1}(0)
\end{CD}
\end{eq}
where we use $\pi$ to denote the unique morphisms from objects of $CC$ to the final object $l^{-1}(0)$. We set
$$
\pi^{n+1}_i=\left\{
\begin{array}{ll}
p_{l^{-1}(n+1)}\circ \pi^n_i&\mbox{\rm for $i<n$}\\
q(\pi,l^{-1}(1))&\mbox{\rm for $i=n$}
\end{array}
\right.
$$
\end{construction}
\begin{problem}
For any $m,n\in\nat$ and a function $f:stn(m)\sr stn(n)$ to construct a morphism $L_f:l^{-1}(n)\sr l^{-1}(m)$ in $CC$. 
\end{problem}
\begin{construction}\rm
\llabel{2015.11.28.constr2}
By induction on $m$. 

For $m=0$ we set $L_f=\pi$.

For $m=1$ we set $L_f=\pi^n_{f(0)}$.

For the successor consider $f:stn(m+1)\sr stn(n)$ and the square (\ref{2015.11.28.eq1}) for $m+1$. We define $L_f$ as the unique morphism such that:
\begin{eq}\llabel{2015.11.30.eq2}
L_f\circ p_{l^{-1}(m+1)}=L_{ii_1^{m,1}\circ f}
\end{eq}
and
\begin{eq}\llabel{2015.11.30.eq3}
L_f\circ q(\pi,l^{-1}(1))=L_{ii_2^{m,1}\circ f}
\end{eq}
where, let us recall, 
$$ii_1^{m,1}:stn(m)\sr stn(m+1)$$
$$ii_2^{m,1}:stn(1)\sr stn(m+1)$$
are the morphism that define the representation
$$stn(m+1)=stn(m)\amalg \{m+1\}$$
\end{construction}
\begin{lemma}
\llabel{2015.11.28.l1}
Let $m,n\in\nat$ and $f:stn(m)\sr stn(n)$. The for any $i=0,\dots, m-1$ one has
$$L_f\circ \pi^m_i=\pi^n_{f(i)}$$
\end{lemma}
\begin{proof}
By induction on $m$.

For $m=0$ there is nothing to prove.

For $m=1$ we need to prove that $L_f\circ \pi^n_0=\pi^n_{f(0)}$. By construction, $\pi^n_0=Id_{l^{-1}(1)}$ and $L_f=\pi^n_{f(0)}$ which implies the goal. 

For the successor $f:stn(m+1)\sr stn(n)$ and the square (\ref{2015.11.28.eq1}) for $m+1$. 

If $i=m$ then $\pi_i^{m+1}=q(\pi,l^{-1}(1))$ and by the construction of $L_f$ we have $L_f\circ q(\pi,l^{-1}(1))=\pi^n_{f(m)}$.

If $i<m$ then $\pi^{m+1}_i=p_{l^{-1}(m+1)}\circ \pi_i^m$. Therefore
$$L_f\circ \pi^{m+1}_i=L_f\circ p_{l^{-1}(m+1)}\circ \pi^m_i=L_{ii^{m,1}_1\circ f}\circ \pi^m_i$$
and by the inductive assumption
$$L_{ii^{m,1}_1\circ f}\circ \pi^m_i=\pi^n_{f(ii^{m,1}_1)}=\pi^n_{f(i)}$$
Lemma is proved.
\end{proof}
\begin{lemma}
\llabel{2015.11.28.l2}
Let $m,n\in\nat$ and $f,g:l^{-1}(n)\sr l^{-1}(m)$ are two morphisms such that for all $i=0,\dots,m-1$ one has
$$f\circ \pi^m_i=g\circ \pi^m_g$$
Then $f=g$.
\end{lemma}
\begin{proof}
By induction on $m$.

For $m=0$ , $l^{-1}(m)$ is a final object and $f=g$.

For $m=1$, $\pi_0^1=Id_{l^{-1}(1)}$ and $f=f\circ Id=g\circ Id=g$.

For the successor the square (\ref{2015.11.28.eq1}) for $m+1$. Since the square is a pull-back square it is sufficient to show that 
$$f\circ p_{l^{-1}(m+1)}=g\circ p_{l^{-1}(m+1)}$$
and
$$f\circ q(\pi,l^{-1}(1))=g\circ q(\pi,l^{-1}(1))$$
The second equality follows from the fact that $q(\pi,l^{-1}(1))=\pi^{m+1}_m$.

The first equality follows by the inductive assumption since for $i=0,\dots,m-1$ we have
$$(f\circ p_{l^{-1}(m+1)})\circ \pi^m_i=f\circ \pi^{m+1}_i$$
and
$$(g\circ p_{l^{-1}(m+1)})\circ \pi^m_i=g\circ \pi^{m+1}_i$$
Lemma is proved.
\end{proof}
\begin{lemma}
\llabel{2015.11.28.l3}
\begin{enumerate}
\item for any $m\in\nat$ one has $L_{Id_{stn(m)}}=Id_{l^{-1}(m)}$,
\item for $k,m,n\in\nat$ and $f:stn(k)\sr stn(m)$, $g:stn(m)\sr stn(n)$ one has
$$L_{f\circ g}=L_g\circ L_f$$
\end{enumerate}
\end{lemma}
\begin{proof}
Both cases follow in a straightforward way from Lemmas \ref{2015.11.28.l2} and \ref{2015.11.28.l1}. 
\end{proof}
We can now provide a construction for Problem \ref{2015.11.24.prob2}.
\begin{construction}\rm
\llabel{2015.11.30.constr1}
We need to construct a Lawvere theory structure on $CC^{op}$, i.e. a functor $L:F\sr CC^{op}$ satisfying the conditions of Definition \ref{2015.11.24.def1}. We define the object part of $L$ as $l^{-1}$. We define the morphism part of $L$ as $L_{Mor}(f)=L_f$. Lemma \ref{2015.11.28.l3} shows that $L$ is a covariant functor to $CC^{op}$.  

The first condition of Definition \ref{2015.11.24.def1} is obvious. The second condition as well (it follows from the axioms of a C-system). To prove the third condition we first apply Lemma \ref{2015.11.28.l4}. It remains to prove that squares of the form (\ref{2015.11.28.eq2}) are push-out squares in $CC^{op}$ or, equivalently, that squares of the form
\begin{eq}\llabel{2015.11.30.eq1}
\begin{CD}
l^{-1}(m+1) @>L_{ii_2^{m,1}}>> l^{-1}(1)\\
@VL_{ii_1^{1,m}}VV @VVV\\
l^{-1}(m) @>>> l^{-1}(0)
\end{CD}
\end{eq}
in $CC$ are pull-back squares. We will do it by showing that the square (\ref{2015.11.30.eq1}) equals to the canonical square of the C-system structure for the pair $(\pi_{l^{-1}(m)}, l^{-1}(1))$. The right hand side vertical morphism, $\pi_{l^{-1}(1)}$,  is a unique morphism from $l^{-1}(1)$ to $l^{-1}(0)$ and since $ft(l^{-1}(1))=l^{-1}(0)$ it equals $p_{l^{-1}(1)}$. 

It remains to show that
\begin{eq}\llabel{2015.12.02.eq1}
L_{ii_1^{m,1}}=p_{l^{-1}(m+1)}
\end{eq}
and
\begin{eq}\llabel{2015.12.02.eq2}
L_{ii_2^{m,1}}=q(\pi_{l^{-1}(m)},l^{-1}(1))
\end{eq}
These equalities follow from the equalities (\ref{2015.11.30.eq2}) and (\ref{2015.11.30.eq3}) for $f=Id_{l^{-1}(m+1)}$ because of Lemma \ref{2015.11.28.l3}(1). The construction is completed.  
\end{construction}

Next we will show that our function on objects extends to a functor from the category of l-bijective C-systems to the category of Lawvere theories. First we need the following lemma.
\begin{lemma}
\llabel{2015.12.08.l3}
Let $H:CC_1\sr CC_2$ be a homomorphism of l-bijective C-systems. The for any $n\in\nat$ and $i=0,\dots,n-1$ one has
\begin{eq}\llabel{2015.12.08.eq3}
H(\pi^n_i)=\pi^n_i
\end{eq}
\end{lemma}
\begin{proof}
Note that since we have $l_2(H(X))=l_1(X)$ both sides of (\ref{2015.12.08.eq3}) are morphisms from $l_2^{-1}(n)$ to $l_2^{-1}(1)$. 

The proof is by induction on $n$.

For $n=0$ there are no equations to prove.

For $n=1$ we have $\pi^1_0=Id_{l^{-1}(1)}$ and the statement of the lemma follows from the identity axiom of the definition of a functor.

For the successor we have two cases. For $i<n$ we have
$$H(\pi^{n+1}_i)=H(p_{l^{-1}(n+1)}\circ \pi_i^n)=H(p_{l^{-1}(n+1)})\circ H(\pi^n_i)=p_{l^{-1}(n+1)}\circ \pi_i^n=\pi^{n+1}_i$$
where the third equality uses the inductive assumption. For $i=n$ we have
$$H(\pi^{n+1}_n)=H(q(\pi,l^{-1}(1)))=q(\pi,l^{-1}(1))=\pi^{n+1}_n$$
Lemma is proved. 
\end{proof} 
\begin{lemma}\llabel{2015.12.08.l2}
Let $G:(CC_1,cs_1)\sr (CC_2,cs_2)$ be a homomorphism of C-systems. Then the functor $G^{op}:CC_1^{op}\sr CC_2^{op}$ is a morphism of Lawvere theories $(CC_1^{op},CL(cs_1))\sr (CC_2^{op},CL(cs_2))$.
\end{lemma}
\begin{proof}
Let $CL(cs_1)=L_1$ and $CL(cs_2)=L_2$. We need to show that $L_1\circ G^{op}=L_2$. The equality between the object components of these functors follows from the fact that a homomorphism of C-systems is compatible with the length functions. 
For the morphism component it is more convenient to consider equivalent equation 
$$L_1'\circ G=L_2'$$
where $L_i':F^{op}\sr CC_i$. Then we have to show that for any $f:stn(m)\sr stn(n)$ one has $G(L_{1,f})=L_{2,f}$. Both sides of this equality are morphisms $l_2^{-1}(n)\sr l_2^{-1}(m)$. By Lemma \ref{2015.11.28.l2} it is sufficient to show that $G(L_{1,f})\circ \pi^n_i=L_{2,f}\circ \pi^n_i$ for all $i=0,\dots,n-1$. We have
$$G(L_{1,f})\circ \pi^n_i=G(L_{1,f})\circ G(\pi^n_i)=G(L_{1,f}\circ \pi^n_i)=G(\pi^n_{f(i)})=\pi^n_{f(i)}$$
where we used Lemma \ref{2015.11.28.l1} and twice Lemma \ref{2015.12.08.l3}. On the other hand
$$L_{2,f}\circ \pi^n_i=\pi^n_{f(i)}$$
again by Lemma \ref{2015.11.28.l1}. This completes the proof of Lemma \ref{2015.12.08.l2}.
\end{proof}
\begin{problem}\llabel{2015.12.08.prob2}
For any universe $U$ to construct a functor $CL_U:Cs_{\nat}(U)\sr LW(U)$.
\end{problem}
\begin{construction}\rm\llabel{2015.12.08.constr3}
The object component of $CL_U$ takes a C-system $(CC,cs)$ to the Lawvere theory $(CC^{op},CL(cs))$ where $CL(cs)$ is defined by Construction \ref{2015.11.30.constr1}.

The morphism component takes a homomorphism $G:(CC_1,cs_1)\sr (CC_2,cs_2)$ to $G^{op}$. It is well defined by Lemma \ref{2015.12.08.l2}. 

The identity and composition axioms are straightforward from the corresponding properties of functor composition and its compatibility with the function $G\mapsto G^{op}$. 
\end{construction}

\subsection{Isomorphism theorem}

\begin{theorem}
\llabel{2015.12.08.th1}
For any universe $U$, Constructions \ref{2015.12.08.constr1} and \ref{2015.12.08.constr3} define mutually inverse isomorphisms between the categories of Lawvere theories in $U$ and l-bijective C-systems in $U$. 
\end{theorem}
\begin{proof}
Let us show first that the object components of the functors $LC_U$ and $CL_U$ are mutually inverse bijections of sets. We write $LC$ for $(LC_U)_{Ob}$ and similarly for $CL$. 

Let us consider first the composition $LC\circ CL$ of the function of Construction \ref{2015.11.24.constr1} with the function of Construction \ref{2015.11.30.constr1}. 

Let $Lw$ be a Lawvere theory structure on $T$ and let
$$Lw'=(LC\circ CL)(Lw)=CL(LC(Lw))$$
We have $Lw=(L:F\sr T)$ and $Lw'=(L':F\sr T)$. We have to prove that $L'(f)=L(f)$. To distinguish $L'$ and $L$ more clearly we will use the notation $L'_f$ for $L'(f)$ as we did in Construction \ref{2015.11.30.constr1}. 

On objects we have
$$L'(n)=l^{-1}(n)=(L^{-1})^{-1}(n)=L(n)$$
Let $f:stn(m)\sr stn(n)$ be a morphism in $F$. 

Since $L$ and $L'$ coincide on objects both morphisms $L(f)$ and $L'_f$ are of the form $L(n)\sr L(m)$ when considered as contravariant functors to $CC=T^{op}$.

The prove is by induction on $m$. 

If $m=0$, $L(m)=L'(m)$ is a final object and any two morphisms with it as a codomain are equal.

If $m=1$ we have by  Constructions \ref{2015.11.30.constr1} and \ref{2015.11.28.constr2}  that 
$$L'(f)=L'_f=\pi^n_{f(0)}$$
We need to show that
$$L(f)=\pi^n_{f(0)}$$
We prove these equalities by induction on $n$.  

If $n=0$ then $f$ exists only if $m=0$ and we have already considered this case.

Let $(l,pt,ft,p,q,s)$ be the components of $LC(L)$.  If $n=1$ then by Construction \ref{2015.11.28.constr1} we have $\pi^n_{f(0)}=\pi^n_{0}=Id_{l^{-1}(1)}$ and since $L$ is a functor and $f=Id_{l^{-1}(0)}$ we have $L(f)=Id_{l^{-1}(1)}$.

For the successor we should consider diagram (\ref{2015.11.28.eq1}). We have that $\pi^{n+1}_{f(0)}$ is given by:
$$\pi^{n+1}_{f(0)}=\left\{
\begin{array}{ll}
p_{l^{-1}(n+1)}\circ \pi^n_{(f(0)}&\mbox{\rm for $f(0)<n$}\\
q(\pi,l^{-1}(1))&\mbox{\rm for $f(0)=n$}
\end{array}
\right.
$$
Assume that $f(0)<n$. Then, by the inductive assumption we have $\pi^n_{(f(0)}=L(f)$. By (\ref{2015.12.02.eq1}) we have $p_{l^{-1}(n+1)}=L(ii_1^{n,1})$ and since $L$ is a functor we get
$$\pi^{n+1}_{f(0)}=L(ii_1^{n,1})\circ \pi^n_{(f(0))}=L(ii_1^{n,1})\circ L(g)$$
where $g:stn(1)\sr stn(n)$ is such that $g(0)=f(0)$. Since $L$ is a contravariant functor we get
$$L(ii_1^{n,1})\circ L(g)=L(g\circ ii_1^{n,1})=L(f)$$
For $f(0)=n$ we have 
$\pi^{n+1}_{f(0)}=q(\pi,l^{-1}(1))=L(ii_2^{n,1})=L(f)$
where the second equality is by (\ref{2015.12.02.eq2}).

We have to consider now the case of the successor of $m$. 

The morphism $L'_f$ for $f:stn(m+1)\sr stn(n)$ is defined in (\ref{2015.11.30.eq2}) and (\ref{2015.11.30.eq3}) as the unique morphism such that 
\begin{eq}%\llabel{2015.11.30.eq2}
L'_f\circ p_{l^{-1}(m+1)}=L'_{ii_1^{m,1}\circ f}
\end{eq}
and
\begin{eq}%\llabel{2015.11.30.eq3}
L'_f\circ q(\pi,l^{-1}(1))=L'_{ii_2^{m,1}\circ f}
\end{eq}
By the inductive assumption we have 
$$L'_{ii_1^{m,1}\circ f}=L(ii_1^{m,1}\circ f)$$
$$L'_{ii_2^{m,1}\circ f}=L(ii_2^{m,1}\circ f)$$
It remains to prove that
\begin{eq}\llabel{2015.12.02.eq5}
L(f)\circ p_{l^{-1}(m+1)}=L(ii_1^{m,1}\circ f)
\end{eq}
and
\begin{eq}\llabel{2015.12.02.eq6}
L(f)\circ q(\pi_{l^{-1}(m)},l^{-1}(1))=L(ii_2^{m,1}\circ f)
\end{eq}
By (\ref{2015.12.02.eq1}) and (\ref{2015.12.02.eq2}) we have 
$$p_{l^{-1}(m+1)}=L'_{ii_1^{1,m}}=L(ii_1^{1,m})$$
$$q(\pi_{l^{-1}(m)},l^{-1})=L'_{ii_2^{m,1}}=L(ii_2^{m,1})$$
Where the second equalities are by the inductive assumption on $m$. This proves equalities (\ref{2015.12.02.eq5}) and (\ref{2015.12.02.eq6}).

Let us consider now the composition $CL\circ LC$. Fix a category $CC$ in $U$. We need to show that for a l-bijective C-system structure $cs$ one has $LC(CL(cs))=cs$. Let $CL(cs)=(L)$ and let
$$cs=(l,pt,ft,p,q,s)$$
and
$$LC(CL(cs))=(l',pt',ft',p',q',s')$$

Then on objects $L=l^{-1}$ and $l'=L^{1-}$. Therefore $l=l'$.

Next $pt=l^{-1}(0)=(l')^{-1}(0)=pt'$. Similarly we see that $ft=ft'$ since on a l-bijective C-system $ft$ is determined by $l$. 

For $X$ such that $l(X)>0$ we have by Construction \ref{2015.11.24.constr1} that $p'_X=L(ii_1^{l(X)-1,1})$. Together with (\ref{2015.12.02.eq1} we obtain
$$p'_X=L(ii_1^{l(X)-1,1})=p_{l^{-1}(l(X)}$$
since the C-system is l-bijective we have $X=l^{-1}(l(X)$ and therefore $p'=p$. 

The morphism $q'$ is defined in Construction \ref{2015.11.24.constr1} as the unique morphism such that (\ref{2015.12.02.eq7}) and (\ref{2015.12.02.eq8}) hold. In our notation these equations take the form:
\begin{eq}
%\llabel{2015.12.02.eq7}
q'(f,X)\circ p'_X=p'_{f^*(X)}\circ f
\end{eq}
and
\begin{eq}
\llabel{2015.12.02.eq9}
q'(f,X)\circ L(ii_2^{l(X)-1,1})=L(ii_2^{l(Y),1})
\end{eq}
We need to check that the same equations hold for $q$. For the first one it follows immediately from the fact that $p'=p$. 

To prove the second one consider equation (\ref{2015.12.02.eq2}). Applying this equation to (\ref{2015.12.02.eq9}) for $q$ and using the fact that of C-systems are l-bijective we get
$$q(f,X)\circ q(\pi_{ft(X)},l^{-1}(1)=q(\pi_Y, l^{-1}(1))$$
which a particular case of the composition axiom for $q$ (see \cite[Definition 2.1(7)]{Csubsystems}). 

This proves that $CL_U$ and $LC_U$ are mutually inverse bijections on the sets of objects of our categories. 

The fact that they give mutually inverse functions on morphisms between each pair of objects is straightforward. Indeed
$$CL_{U,Mor}(G)=G^{op}$$
and
$$LC_{U,Mor}(G)=G^{op}$$
as functors and since the mappings from morphisms of Lawvere theories and homomorphisms of C-systems to functors between the corresponding categories are injective we see that $CL_{U,Mor}$ and $LC_{U,Mor}$ are mutually inverse bijections. 
\end{proof}

{\em Acknowledgements:} This material is based on research sponsored by The United States Air Force Research Laboratory under agreement number FA9550-15-1-0053. The US Government is authorized to reproduce and distribute reprints for Governmental purposes notwithstanding  any copyright notation thereon.

The views and conclusions contained herein are those of the author and should not be interpreted as necessarily representing the official policies or endorsements, either expressed or implied, of the United States Air Force Research Laboratory, the U.S. Government or Carnegie Melon University.

%\bibliography{../../../alggeom}
%\bibliographystyle{plain}

\def\cprime{$'$}

\end{document}